\documentclass[11pt]{article}
\usepackage{geometry,amsmath,amsthm,amssymb,latexsym}

\usepackage[svgnames]{xcolor}
\usepackage[colorlinks=true,linkcolor=MidnightBlue,citecolor=MidnightBlue,urlcolor=MidnightBlue,pdfpagelabels=false]{hyperref}

\usepackage{thmtools}
\theoremstyle{plain}
  \declaretheorem[numberwithin=section]{theorem}
  \declaretheorem[numberlike=theorem]{corollary}
  \declaretheorem[numberlike=theorem]{proposition}
  \declaretheorem[numberlike=theorem]{lemma}
  \declaretheorem[numberlike=theorem]{conjecture}
  
\theoremstyle{definition}
  
  \declaretheorem[numberlike=theorem]{example}

\newcommand{\comma}{{,}}
\newcommand{\md}{\mathrm{d}}

\begin{document}

\title{Weakening the Legendre Conjecture}

\author{
Marc Chamberland\thanks{\textit{Email:} \texttt{chamberl@math.grinnell.edu}}\\
Grinnell College
\and
Armin Straub\thanks{\textit{Email:} \texttt{straub@southalabama.edu}}\\
University of South Alabama
}

\date{February 25, 2026}

\maketitle

\begin{abstract}
  The world of primes has many gaps between evidence and theorems. Here, we
  review Legendre's conjecture on primes between consecutive squares and
  recent progress on the weaker question of primes between consecutive larger
  powers. Assuming the Riemann hypothesis (RH), we observe that a recent
  result of Emanuel Carneiro, Micah Milinovich and Kannan Soundararajan,
  combined with a large-scale computation by Jonathan Sorenson and Jonathan
  Webster, implies the existence of primes between $x^{2 + \delta}$ and $(x +
  1)^{2 + \delta}$ for all real $x \geq 1$ when $\delta \geq 1 / 4$.
  For smaller values of $\delta > 0$, we provide an explicit bound $x_0 = x_0
  (\delta)$ such that primes exist in these intervals whenever $x \geq
  x_0$ (again assuming RH). We conclude with an application to Mills-type
  prime-generating constants.
\end{abstract}

\section{Introduction.}

Adrien-Marie Legendre famously conjectured that there is always a prime
between consecutive squares. In 1912, this conjecture was featured by Edmund
Landau at the International Congress of Mathematicians as one of four
``unattackable'' problems in number theory (the other three problems being the
Goldbach conjecture, the twin prime conjecture, and the conjecture that there
are infinitely many primes of the form $n^2 + 1$).

\begin{conjecture}[Legendre]
  There is a prime between $n^2$ and $(n + 1)^2$ for all positive integers
  $n$.
\end{conjecture}

This conjecture is usually stated for positive integers $n$ but it appears to
also hold when $n$ is replaced by a positive, real number. In other words,
conjecturally, every interval $[x^2, (x + 1)^2]$, where $x \geq 1$ is a
real number, contains at least one prime. Throughout, we will focus on this
slight strengthening of Legendre's conjecture.

While Legendre's conjecture has withstood proof for more than a century since
Landau's endorsement, it is natural to wonder what can be said about weaker
versions. For instance, it is known that there is a prime between $n^3$ and
$(n + 1)^3$ for all sufficiently large $n$. In fact, this is known to be true
for $n \geq \exp (\exp (32.537))$ due to recent work of Michaela
Cully-Hugill \cite[Theorem~4.3]{cully-hugill-phd}, improving a similar
earlier bound of Adrian Dudek \cite{dudek-primes-cubes}. Stronger results
are available if we assume that the Riemann hypothesis (RH) is true (this is
because, as is recalled at the beginning of Section~\ref{sec:legendre:weak},
RH is equivalent to certain estimates for the number of primes less than some
$x$). A goal of this article is to illustrate this explicitly in
Theorem~\ref{thm:legendre:a:rh:intro} below, which follows from combining a
recent result of Emanuel Carneiro, Micah Milinovich and Kannan Soundararajan
\cite{cms-primes-fourier} with a large-scale computation by Jonathan
Sorenson and Jonathan Webster \cite{sw-legendre}.

\begin{theorem}
  \label{thm:legendre:a:rh:intro}Suppose the Riemann hypothesis is true, and
  let $\delta > 0$.
  \begin{enumerate}
    \item \label{i:legendre:a:rh:intro:a}If $\delta \geq 1 / 4$, then the
    interval $[x^{2 + \delta}, (x + 1)^{2 + \delta}]$ contains a prime for all
    $x \geq 1$.
    
    \item If $\delta < 1 / 4$, then the interval $[x^{2 + \delta}, (x + 1)^{2
    + \delta}]$ contains a prime for all
    \begin{equation*}
      x \geq [\beta (\log (\beta) + \log (\log (\beta)))]^{\beta}, \quad
       \beta = 2 / \delta .
    \end{equation*}
  \end{enumerate}
\end{theorem}

With more work, this result can be sharpened. For instance, as indicated at
the end of Section~\ref{sec:legendre:rh}, the bound for $\delta$ in part
\ref{i:legendre:a:rh:intro:a} can be improved to $0.2253$. On the other hand,
obtaining a result like $\delta \geq 1 / 5$ in part
\ref{i:legendre:a:rh:intro:a} seems out of reach with the current state of
computation and theory.

Such Legendre-type results have an appealing application to prime-generating
constants originating with William Mills \cite{mills-primegen}, who famously
proved that there exists a real number $A$ such that $\lfloor A^{3^n} \rfloor$
is a prime for every positive integer $n$. Mills' proof is based on a result
of Albert Ingham \cite{ingham-primes-diff} from which he deduces that there
is always a prime between $n^3$ and $(n + 1)^3$ for all sufficiently large
$n$. By the special case $\delta = 1$ of
Theorem~\ref{thm:legendre:a:rh:intro}, which was previously proved by Chris
Caldwell and Yuanyou Cheng \cite{cc-mills}, this is in fact true for all
positive integers $n$ if we are willing to assume the Riemann hypothesis.
Caldwell and Cheng use that fact to show that, assuming RH, the number
\begin{equation}
  A = 1.306377883863080690468614492602 \ldots \label{eq:mills:ct}
\end{equation}
is the smallest number with the property that $\lfloor A^{3^n} \rfloor$ is a
prime for all $n$. Without assuming RH, Christian Elsholtz
\cite{elsholtz-mills} has produced unconditional variants that generate only
primes. Namely, to high precision, Elsholtz unconditionally gives constants
$B$ and $C$ such that $\left\lfloor B^{10^{10^n}} \right\rfloor$ and
$\left\lfloor C^{3^{13^n}} \right\rfloor$ are prime for all $n$. In
Section~\ref{sec:mills}, assuming RH, we show that (a minor strengthening of)
Theorem~\ref{thm:legendre:a:rh:intro} allows us to determine, for any fixed
$\alpha \geq 2.25$, the smallest constant $A$ such that $\lfloor
A^{\alpha^n} \rfloor$ is a prime for all $n$.

In Section~\ref{sec:legendre:weak}, we begin by demonstrating how the Riemann
hypothesis readily implies that there is a prime between $x^{2 + \delta}$ and
$(x + 1)^{2 + \delta}$ for all sufficiently large $x$. The explicit bounds
provided by Theorem~\ref{thm:legendre:a:rh:intro} on $x$ are then discussed in
Section~\ref{sec:legendre:rh}. In Section~\ref{sec:legendre:smaller}, we
briefly consider the question of reducing the exponents of $x$ and $x + 1$ to
values less than $2$, in which case one can again ask whether there are primes
between such smaller powers. In that case, it is truly necessary to restrict
to values of $x$ that are sufficiently large. The resulting conjecture,
however, remains open even assuming the Riemann hypothesis. We conclude in
Section~\ref{sec:mills} with an application to prime-generating constants
originating with Mills \cite{mills-primegen}.

In the remainder of this introduction, we compare
Theorem~\ref{thm:legendre:a:rh:intro} with various existing results that do
not rely on the Riemann hypothesis. For instance, it is known that
Theorem~\ref{thm:legendre:a:rh:intro} holds for large enough values of $\alpha
= 2 + \delta$ unconditionally. Specifically, Dudek \cite{dudek-primes-cubes}
has proved, without relying on RH, that if $\alpha \geq 5 \cdot 10^9$,
then there is a prime between $n^{\alpha}$ and $(n + 1)^{\alpha}$ for all $n
\geq 1$. Caitlin Mattner \cite{mattner-bsc} subsequently improved this
result to hold for all $\alpha \geq 1 \comma 438 \comma 989$. Very
recently, Michaela Cully-Hugill and Daniel Johnston
\cite{chj-primes-powers2} have unconditionally proved that, for $\alpha =
90$, there is always a prime between $n^{\alpha}$ and $(n + 1)^{\alpha}$ for
all $n \geq 1$. This improves earlier results
\cite{cully-hugill-primes-powers} and \cite{chj-primes-powers} that showed
the same conclusion for $\alpha = 155$ and $\alpha = 140$.

Let $p_n$ denote the $n$th prime. Numerous results can be found in the
literature that prove, for fixed $\theta$,
\begin{equation}
  p_{n + 1} - p_n < p_n^{\theta} \label{eq:primegaps:theta}
\end{equation}
for all sufficiently large $n$. The case $\theta = 1$ was postulated by Joseph
Bertrand in 1845 and proved by Pafnuty Chebyshev in 1852 (in this case
\eqref{eq:primegaps:theta} holds for all $n \geq 1$). If $\theta < 1$ the
inequalities \eqref{eq:primegaps:theta} are equivalent to the intervals $[x -
x^{\theta}, x]$ containing at least one prime for all sufficiently large $x$.
The first such result was due to Guido Hoheisel \cite{hoheisel-primes-diff}
who showed that one can take $\theta = 32999 / 33000$. For a history of
results of this kind, we refer to \cite{bh-primes-diff}, where this is
improved to $\theta = 0.535$. At present, the strongest such result is due to
Roger Baker, Glyn Harman and J{\'a}nos Pintz \cite{bhp-primes-diff2} who
prove that $\theta = 0.525$ is possible. Assuming the Riemann hypothesis, it
is known that we can choose any $\theta > 1 / 2$, while a conjecture of Harald
Cram\'er \cite{cramer-primes-diff} implies that, in fact, any $\theta > 0$
is possible. We refer to Section~\ref{sec:legendre:smaller} for this
conjecture on the true nature of the prime gaps $p_{n + 1} - p_n$. Here, note
that results of the type \eqref{eq:primegaps:theta} yield weak versions of
Legendre's conjecture as follows.

\begin{proposition}
  \label{prop:theta:alpha}Fix $0 < \theta < 1$ and let $\alpha = \frac{1}{1 -
  \theta}$. If the inequality \eqref{eq:primegaps:theta} holds for all
  sufficiently large $n$ then the intervals $[x^{\alpha}, (x + 1)^{\alpha}]$
  contain a prime for all sufficiently large $x$.
\end{proposition}

\begin{proof}
  Suppose that $[x^{\alpha}, (x + 1)^{\alpha}]$ does not contain a prime, and
  choose $n$ so that $p_n < x^{\alpha} < (x + 1)^{\alpha} < p_{n + 1}$. In
  that case,
  \begin{equation*}
    \frac{p_{n + 1} - p_n}{p_n^{\theta}} > \frac{(x + 1)^{\alpha} -
     x^{\alpha}}{x^{\alpha \theta}} > \frac{\alpha x^{\alpha - 1}}{x^{\alpha
     \theta}} = \alpha x^{\alpha (1 - \theta) - 1} = \alpha > 1,
  \end{equation*}
  which contradicts \eqref{eq:primegaps:theta}.
\end{proof}

For $\theta = 0.525$, we have $\frac{1}{1 - \theta} \approx 2.1053$. Hence,
the result by Baker, Harman and Pintz \cite{bhp-primes-diff2} implies the
following.

\begin{corollary}[Baker, Harman, Pintz \cite{bhp-primes-diff2}]
  \label{cor:bhp}If $\alpha \geq 2.106$, then there exists $x_0$ such
  that the intervals $[x^{\alpha}, (x + 1)^{\alpha}]$ contain a prime for all
  $x \geq x_0$.
\end{corollary}

Baker, Harman and Pintz \cite{bhp-primes-diff2} remark that ``with enough
effort'' one could work out an explicit value for the lower bound $x_0$.
However, due to the intricate techniques involved, both the required effort as
well as the resulting bound would likely be enormous. On the other hand, as
seen in Theorem~\ref{thm:legendre:a:rh:intro}, if we assume RH then it is not
too hard to obtain explicit lower bounds for all $\alpha > 2$.

\section{Weak versions of the Legendre conjecture.}\label{sec:legendre:weak}

Let $\pi (x)$ denote the prime counting function and let $\operatorname{Li} (x)$ be
the logarithmic integral
\begin{equation*}
  \operatorname{Li} (x) = \int_2^x \frac{\md t}{\log (t)} .
\end{equation*}
It is well-known that the Riemann hypothesis is equivalent to having the
bounds, as $x \rightarrow \infty$,
\begin{equation}
  \pi (x) = \operatorname{Li} (x) + O (x^{1 / 2 + \varepsilon}) \label{eq:rh:li}
\end{equation}
for all fixed $\varepsilon > 0$. The exponent $1 / 2$ in this bound directly
reflects the fact that the RH predicts all non-trivial zeros of the Riemann
zeta function to have real part $1 / 2$. Unconditionally, it is not even known
whether the exponent $1 / 2$ can be replaced with $\sigma$ for any $\sigma <
1$ (corresponding to the fact that there is no value $\sigma < 1$ for which we
can currently prove that all non-trivial zeros of the Riemann zeta function
have real part less than $\sigma$). We remark that Helge von Koch
\cite{koch-riemann} further proved that the Riemann hypothesis is equivalent
to
\begin{equation}
  \pi (x) = \operatorname{Li} (x) + O \left(\sqrt{x} \log (x) \right) .
  \label{eq:rh:vonkoch}
\end{equation}
(Note that inequality \eqref{eq:rh:vonkoch} is clearly stronger than the bound
\eqref{eq:rh:li} for any $\varepsilon > 0$.)

We now observe that any bound of the form \eqref{eq:rh:li} implies that there
are primes between corresponding consecutive powers.

\begin{lemma}
  \label{lem:eps:a}Suppose the bound \eqref{eq:rh:li} holds for a certain
  $\varepsilon \in \left(0, \frac{1}{2} \right)$ and let $\alpha > \frac{2}{1
  - 2 \varepsilon}$. Then the interval $[x^{\alpha}, (x + 1)^{\alpha}]$
  contains a prime for all sufficiently large $x$.
\end{lemma}

\begin{proof}
  Suppose that $\alpha > 2 / (1 - 2 \varepsilon)$. Since $O ((x + 1)^{\beta})
  = O (x^{\beta})$ as $x \rightarrow \infty$ for any $\beta > 0$, the bound
  \eqref{eq:rh:li} implies that
  \begin{eqnarray*}
    \pi ((x + 1)^{\alpha}) - \pi (x^{\alpha}) & = & \operatorname{Li} ((x +
    1)^{\alpha}) - \operatorname{Li} (x^{\alpha}) + O (x^{\alpha (1 / 2 +
    \varepsilon)})\\
    & = & \int_{x^{\alpha}}^{(x + 1)^{\alpha}} \frac{\md t}{\log (t)} + O
    (x^{\alpha (1 / 2 + \varepsilon)}) .
  \end{eqnarray*}
  On the other hand,
  \begin{equation*}
    \int_{x^{\alpha}}^{(x + 1)^{\alpha}} \frac{\md t}{\log (t)} > \frac{(x
     + 1)^{\alpha} - x^{\alpha}}{\alpha \log (x + 1)} > \frac{x^{\alpha -
     1}}{\log (x + 1)},
  \end{equation*}
  where we used $(x + 1)^{\alpha} > x^{\alpha} + \alpha x^{\alpha - 1}$ for
  the final inequality.
  
  We now note that the inequality $\alpha - 1 > \alpha (1 / 2 + \varepsilon)$
  is equivalent to our assumption that $\alpha > 2 / (1 - 2 \varepsilon)$.
  Therefore, put together, we have
  \begin{equation*}
    \pi ((x + 1)^{\alpha}) - \pi (x^{\alpha}) > \frac{x^{\alpha - 1}}{\log
     (x)} (1 + o (1)) .
  \end{equation*}
  In particular, we have $\pi ((x + 1)^{\alpha}) - \pi (x^{\alpha}) > 1$ for
  sufficiently large $x$.
\end{proof}

If the Riemann hypothesis is true, then we can choose any $\varepsilon > 0$ in
Lemma~\ref{lem:eps:a}. It follows that the conclusion is true for all $\alpha
> 2$ and we obtain the following.

\begin{corollary}
  \label{cor:legendre:large:rh}Suppose the Riemann hypothesis is true, and let
  $\alpha > 2$ be fixed. Then the interval $[x^{\alpha}, (x + 1)^{\alpha}]$
  contains a prime for all sufficiently large $x$.
\end{corollary}

Let us note that, using a strong result of Ingham \cite{ingham-primes-diff},
the Riemann hypothesis in Corollary~\ref{cor:legendre:large:rh} can be
replaced with the Lindel{\"o}f hypothesis which is known to be weaker than RH.
Namely, Ingham proves that if we have
\begin{equation*}
  \zeta \left(\frac{1}{2} + i t \right) = O (t^c), \quad t \rightarrow
   \infty,
\end{equation*}
for some constant $c > 0$, then the inequalities $p_{n + 1} - p_n <
p_n^{\theta}$ in \eqref{eq:primegaps:theta} hold for large enough $n$ if
$\theta > \frac{1 + 4 c}{2 + 4 c}$. The Lindel{\"o}f hypothesis posits that we
can choose $c$ to be arbitrarily close to $0$. If true, it follows that any
$\theta > \frac{1}{2}$ can be chosen in \eqref{eq:primegaps:theta}. Thus,
Proposition~\ref{prop:theta:alpha} implies
Corollary~\ref{cor:legendre:large:rh} with the Lindel{\"o}f hypothesis in
place of the Riemann hypothesis.

\section{Explicit bounds under RH.}\label{sec:legendre:rh}

The error term in \eqref{eq:rh:vonkoch} suggests the length of short intervals
that are guaranteed to contain a prime. At present, the strongest known result
is the following due to Carneiro, Milinovich and Soundararajan
\cite{cms-primes-fourier}.

\begin{theorem}[Carneiro, Milinovich, Soundararajan
\cite{cms-primes-fourier}]
  \label{thm:cms}Suppose the Riemann hypothesis is true. Then, for all $x
  \geq 4$, there is a prime in $\left[ x, x + \frac{22}{25} \sqrt{x} \log
  (x) \right]$.
\end{theorem}

Theorem~\ref{thm:cms} improves intervals established earlier by different
techniques, for instance, by Olivier Ramar\'e and Yannick Saouter
\cite{rs-primes-int}, Dudek \cite{dudek-riemann-gaps} as well as Dudek,
Lo{\"\i}c Greni\'e and Giuseppe Molteni \cite{dgm-primes-rh}. From
Theorem~\ref{thm:cms}, we can readily conclude the following weak version of
Legendre's conjecture. We note that, given $\alpha > 2$, the inequality
\eqref{eq:bound:n:simpl} holds for sufficiently large $x$ (because $\alpha / 2
- 1 > 0$). We are using a slightly smaller interval than $[x^{\alpha}, (x +
1)^{\alpha}]$ in the conclusions because that is convenient for the
application to prime-generating constants in Lemma~\ref{lem:prc:legendre}.

\begin{corollary}
  \label{cor:legendre:a:rh}Suppose the Riemann hypothesis is true, and let
  $\alpha > 2$ be fixed. If
  \begin{equation}
    x^{\alpha / 2 - 1} \geq \frac{22}{25} \log (x), \quad x \geq 1,
    \label{eq:bound:n:simpl}
  \end{equation}
  then there is a prime in the interval $[x^{\alpha}, (x + 1)^{\alpha} - 1)$.
\end{corollary}

\begin{proof}
  The statement is clearly true for $x = 1$. In the sequel, we therefore
  assume $x > 1$. Observe that the interval $[x^{\alpha}, (x + 1)^{\alpha} -
  1)$ has length exceeding $\alpha x^{\alpha - 1} \geq 2 x > 2$. Hence,
  since the first prime gap larger than $2$ occurs between $7$ and $11$, we
  may further assume $x^{\alpha} > 7$. We set $y = x^{\alpha}$ with the
  intention of applying Theorem~\ref{thm:cms} (with $y$ in place of $x$; note
  that, by our assumptions, the condition $y \geq 4$ is satisfied) and
  observe that
  \begin{equation*}
    y + \frac{22}{25} \sqrt{y} \log (y) = x^{\alpha} + \frac{22}{25} \alpha
     x^{\alpha / 2} \log (x) < (x + 1)^{\alpha} - 1
  \end{equation*}
  if and only if
  \begin{equation}
    (x + 1)^{\alpha} - 1 - x^{\alpha} > \frac{22}{25} \alpha x^{\alpha / 2}
    \log (x) . \label{eq:bound:n}
  \end{equation}
  Since the left-hand side satisfies $(x + 1)^{\alpha} - 1 - x^{\alpha} >
  \alpha x^{\alpha - 1}$, it is sufficient (though not necessary) for
  \eqref{eq:bound:n} that
  \begin{equation*}
    x^{\alpha - 1} \geq \frac{22}{25} x^{\alpha / 2} \log (x),
  \end{equation*}
  which is equivalent to the inequality \eqref{eq:bound:n:simpl}. For those
  $x$ satisfying \eqref{eq:bound:n:simpl}, it therefore follows from
  Theorem~\ref{thm:cms} that the interval $[x^{\alpha}, (x + 1)^{\alpha} - 1)$
  is guaranteed to contain a prime, as claimed.
\end{proof}

The following result provides the explicit bound \eqref{eq:bound:n:expl} on
$x$ that implies the inequality~\eqref{eq:bound:n:simpl}.

\begin{lemma}
  \label{lem:bound:n:expl}Suppose that $2 < \alpha \leq 2 \left(1 +
  \frac{1}{e} \right) \approx 2.736$ and define $\beta = 2 / (\alpha - 2)$.
  Then the inequality~\eqref{eq:bound:n:simpl} holds for all
  \begin{equation}
    x \geq [\beta (\log (\beta) + \log (\log (\beta)))]^{\beta} .
    \label{eq:bound:n:expl}
  \end{equation}
\end{lemma}

\begin{proof}
  The inequality~\eqref{eq:bound:n:simpl} takes the form $x \geq (C \log
  (x))^{\beta}$ with $C = 22 / 25$. It follows directly from
  \begin{equation*}
    \frac{\md}{\md x}  \frac{x}{\log (x)^{\beta}} = \frac{\log (x) -
     \beta}{\log (x)^{\beta + 1}}
  \end{equation*}
  that $x / \log (x)^{\beta}$ is an increasing function for $x \geq
  e^{\beta}$. On the other hand, the bounds on $\alpha$ imply that $\beta
  \geq e$. In particular, $\log (\log (\beta)) \geq 0$ so that it
  follows from \eqref{eq:bound:n:expl} that $x \geq \beta^{\beta}
  \geq e^{\beta}$. It therefore suffices to show that the
  inequality~\eqref{eq:bound:n:simpl} holds for
  \begin{equation*}
    x = [\beta (\log (\beta) + \log (\log (\beta)))]^{\beta} .
  \end{equation*}
  With this value for $x$ fixed,
  \begin{eqnarray*}
    \log (x) & = & \beta \log (\beta (\log (\beta) + \log (\log (\beta))))\\
    & = & \beta \left(\log (\beta) + \log (\log (\beta)) + \log \left(1 +
    \frac{\log (\log (\beta))}{\log (\beta)} \right) \right) .
  \end{eqnarray*}
  It follows that
  \begin{equation*}
    x^{1 / \beta} - \frac{22}{25} \log (x) = \frac{\beta}{25} \left(3 (\log
     (\beta) + \log (\log (\beta))) - 22 \log \left(1 + \frac{\log (\log
     (\beta))}{\log (\beta)} \right) \right),
  \end{equation*}
  and it only remains to observe that the right-hand side is positive (at
  least for $\beta \geq e$). This is left as a calculus exercise (the
  coefficients $3$ and $22$ are not optimal).
\end{proof}

Sorenson and Webster \cite{sw-legendre} recently described an algorithm to
efficiently test a slightly stronger version of Legendre's conjecture due to
Ludvig Oppermann, namely that there are primes between $n^2$ and $n (n + 1)$
as well as between $n (n + 1)$ and $(n + 1)^2$ for all integers $n > 1$. Using
that algorithm they verify that this is true for all $n \leq N$ where $N
= 7.05 \cdot 10^{13}$. The following observation allows us to conclude from
this computation that the intervals $[x^{\alpha}, (x + 1)^{\alpha}]$ contain a
prime provided that $1 \leq x^{\alpha / 2} < N$. In particular, letting
$N \rightarrow \infty$, we see that Oppermann's conjecture implies that, for
every $\alpha \geq 2$ and $x \geq 1$, there is at least one prime in
the interval $[x^{\alpha}, (x + 1)^{\alpha}]$.

\begin{proposition}
  \label{prop:oppermann:legendre}Let $N$ be an integer and suppose that the
  intervals $[n^2, n (n + 1)]$ and $[n (n + 1), (n + 1)^2]$ each contain a
  prime for all $n \in \{ 1, 2, \ldots, N \}$. Then, for every $\alpha
  \geq 2$ and $x \geq 1$, the interval $[x^{\alpha}, (x +
  1)^{\alpha}]$ contains a prime provided that $x^{\alpha / 2} < N$.
\end{proposition}

\begin{proof}
  Consider a specific interval $[x^{\alpha}, (x + 1)^{\alpha}]$ for $x
  \geq 1$ and $\alpha \geq 2$. If $y = x^{\alpha / 2}$ then $[y^2,
  (y + 1)^2]$ is contained in $[x^{\alpha}, (x + 1)^{\alpha}]$. It therefore
  suffices to show that every interval $[y^2, (y + 1)^2]$ contains a prime
  provided that $1 \leq y < N$.
  
  Set $m = \lfloor y \rfloor$ so that $m^2 \leq y^2 < (m + 1)^2 \leq
  (y + 1)^2$. First, let us assume that $m (m + 1) \geq y^2$. In that
  case, we have $[m (m + 1), (m + 1)^2] \subseteq [y^2, (y + 1)^2]$. Hence
  both intervals contain a prime provided that $m \leq N$. Second, we
  assume that $m (m + 1) < y^2$. In that case,
  \begin{equation*}
    (y + 1)^2 = y^2 + (2 y + 1) > m (m + 1) + (2 m + 1) = (m + 1) (m + 2) -
     1.
  \end{equation*}
  Consequently, we find that $[(m + 1)^2, (m + 1) (m + 2) - 1] \subset [y^2,
  (y + 1)^2]$. Hence both intervals contain a prime provided that $m + 1
  \leq N$.
\end{proof}

Equipped with Corollary~\ref{cor:legendre:a:rh} as well as
Lemma~\ref{lem:bound:n:expl} and Proposition~\ref{prop:oppermann:legendre}, we
complete the proof of Theorem~\ref{thm:legendre:a:rh:intro} stated in the
introduction. Since it is useful for the application in
Section~\ref{sec:mills}, we will actually prove
Theorem~\ref{thm:legendre:a:rh:intro} with the intervals $[x^{2 + \delta}, (x
+ 1)^{2 + \delta}]$ replaced with the slightly smaller intervals $[x^{2 +
\delta}, (x + 1)^{2 + \delta} - 1)$.

\begin{proof}[Proof of Theorem~\ref{thm:legendre:a:rh:intro}]
  If $\delta < 1 / 4$, then the result follows from combining
  Corollary~\ref{cor:legendre:a:rh} with Lemma~\ref{lem:bound:n:expl}. In the
  remainder, we therefore assume $\delta \geq 1 / 4$. Note that if the
  interval $[x^{\alpha}, (x + 1)^{\alpha} - 1)$ does not contain a prime for
  some $x \geq 1$, then, for $\beta < \alpha$, the interval $[y^{\beta},
  (y + 1)^{\beta} - 1)$ with $y = x^{\alpha / \beta} \geq 1$ also does
  not contain a prime. It therefore suffices to consider the case $\delta = 1
  / 4$.
  
  Let $\delta = 1 / 4$ and, hence, $\alpha = 2 + \delta = 9 / 4$. From
  Lemma~\ref{lem:bound:n:expl} with $\beta = 2 / \delta = 8$ applied to
  Corollary~\ref{cor:legendre:a:rh}, we find that there is a prime in
  $[x^{\alpha}, (x + 1)^{\alpha} - 1)$ provided that $x \geq [\beta (\log
  (\beta) + \log (\log (\beta)))]^{\beta} \approx 6.55 \cdot 10^{10}$. On the
  other hand, the computations by Sorenson and Webster \cite{sw-legendre}
  allow us to apply Proposition~\ref{prop:oppermann:legendre} with $N = 7.05
  \cdot 10^{13}$ to conclude that all intervals $[y^2, (y + 1)^2]$ contain a
  prime provided that $1 \leq y < N$. If $y = x^{\alpha / 2}$ then $[y^2,
  (y + 1)^2]$ is contained in $[x^{\alpha}, (x + 1)^{\alpha} - 1)$ provided
  that $x$ is large enough. For $\alpha = 9 / 4$ it is sufficient that $x
  \geq 2$. Since smaller values of $x$ are readily verified directly, it
  follows that the intervals $[x^{\alpha}, (x + 1)^{\alpha} - 1)$ contain a
  prime provided that $x^{\alpha / 2} < N$. In the present case, this bound on
  $x$ takes the form $x < N^{2 / \alpha} \approx 2.04 \cdot 10^{12}$. This
  exceeds $6.55 \cdot 10^{10}$ and thus completes the proof.
\end{proof}

Clearly, this final argument is not sharp. Indeed, the same argument goes
through provided that $[\beta (\log (\beta) + \log (\log (\beta)))]^{\beta} <
N^{2 / \alpha}$ or, equivalently,
\begin{equation*}
  [\beta (\log (\beta) + \log (\log (\beta)))]^{\beta + 1} < N.
\end{equation*}
Using the value $N = 7.05 \cdot 10^{13}$ from \cite{sw-legendre}, this
allows us to lower the bound slightly for $\delta$ in
Theorem~\ref{thm:legendre:a:rh:intro} from $1 / 4$ to $0.2275$. By replacing
the bound~\eqref{eq:bound:n:expl} with the
inequality~\eqref{eq:bound:n:simpl}, we can further reduce $\delta$ to
$0.2253$.

\section{Primes between consecutive smaller
powers.}\label{sec:legendre:smaller}

Inspired by Legendre's conjecture, one may wonder whether there are primes
between consecutive smaller powers. More precisely, given $\alpha > 1$, what
can we say about primes between $x^{\alpha}$ and $(x + 1)^{\alpha}$? Here, we
take $\alpha > 1$ since if $\alpha = 1$ then it is trivially the case that
there are arbitrarily large $x$ such that the intervals $[x^{\alpha}, (x +
1)^{\alpha}]$ contain no prime.

\begin{example}
  A quick numerical search suggests that, for integers $n$, there is always a
  prime between $n^{3 / 2}$ and $(n + 1)^{3 / 2}$ provided that $n > 1051$.
  Indeed, it appears that the interval $[n^{3 / 2}, (n + 1)^{3 / 2}]$ does not
  contain a prime precisely if $n$ is one of 10, 20, 24, 27, 32, 65, 121, 139,
  141, 187, 306, 321, 348, 1006, 1051. This is sequence A144140 on the OEIS
  \cite{sloane-oeis} where we find it conjectured that our computed list of
  exceptional values is complete.
\end{example}

In general, the following is expected to be true.

\begin{conjecture}
  \label{conj:legendre:a:small}Let $\alpha > 1$ be fixed. Then there is a
  prime between $x^{\alpha}$ and $(x + 1)^{\alpha}$ for all sufficiently large
  $x$.
\end{conjecture}

This conjecture is another way to state that it is expected that $p_{n + 1} -
p_n = O (p_n^{\delta})$ for any $\delta > 0$. While this is open, even
assuming the Riemann hypothesis, stronger conjectures have been put forth.
Specifically, Cram\'er \cite{cramer-primes-diff} conjectured in 1936 that
gaps between consecutive primes are remarkably small in the following sense.

\begin{conjecture}[Cram\'er \cite{cramer-primes-diff}]
  \label{conj:cramer}As $n \rightarrow \infty$, we have $p_{n + 1} - p_n = O
  (\log^2 (p_n))$.
\end{conjecture}

Due to the extensive computations by Tom{\'a}s Oliveira~e Silva, Siegfried
Herzog and SilvioPardi \cite[Section~2.2.1]{ohp-primegaps} in the context of
verifying the Goldbach conjecture, we know that $p_{n + 1} - p_n < \log^2
(p_n)$ for all primes $11 \leq p_n \leq 4 \cdot 10^{18}$. On the
other hand, Legendre's conjecture implies that $p_{n + 1} - p_n = O \left(\sqrt{p_n} \right)$ while, assuming the Riemann hypothesis, Cram\'er
\cite{cramer-primes-diff} proved that $p_{n + 1} - p_n = O \left(\sqrt{p_n}
\log (p_n) \right)$.

\section{Prime-generating constants.}\label{sec:mills}

The constructions of Mills-type prime-generating constants in
\cite{mills-primegen}, \cite{cc-mills} and \cite{elsholtz-mills},
mentioned in the introduction, are based on the following lemma which connects
these with Legendre-type results about the existence of primes between
consecutive powers.

\begin{lemma}
  \label{lem:prc:legendre}Let $\alpha > 1$. Suppose that the intervals
  $[n^{\alpha}, (n + 1)^{\alpha} - 1)$ contain a prime for all $n \geq
  n_0$. Then there exists a real number $A$ such that $\lfloor A^{\alpha^n}
  \rfloor$ is a prime for every positive integer $n$.
\end{lemma}

The general construction underlying this lemma is essentially given by Ivan
Niven \cite{niven-primegen}. Since it is rather simple and instrumental for
the subsequent examples, we include a proof below. We first observe that the
intervals $[n^{\alpha}, (n + 1)^{\alpha} - 1)$ in the lemma are equivalent to
the intervals $[n^{\alpha}, (n + 1)^{\alpha}]$ if $\alpha \geq 2$ is an
integer (since in that case neither $(n + 1)^{\alpha}$ nor $(n + 1)^{\alpha} -
1$ are prime). On the other hand, we note that, with the present application
in mind, our proof of Theorem~\ref{thm:legendre:a:rh:intro} actually shows
that, assuming the RH, the slightly smaller intervals $[x^{\alpha}, (x +
1)^{\alpha} - 1)$ contain a prime for all $x \geq 1$ provided that
$\alpha \geq 2.25$. For $\alpha > 2$ these intervals contain a prime for
all sufficiently large $x$, and this conclusion continues to hold for all
$\alpha > 1$ if Conjecture~\ref{conj:legendre:a:small} is true.

\begin{proof}[Proof of Lemma~\ref{lem:prc:legendre}]
  Pick $q_1$ to be any prime larger than $n_0$. By assumption, we can then
  select primes $q_2, q_3, \ldots$ in such a way that $q_{n + 1} \in
  [q_n^{\alpha}, (q_n + 1)^{\alpha} - 1)$. Note that we have
  \begin{equation*}
    q_n^{\alpha} \leq q_{n + 1} < q_{n + 1} + 1 < (q_n + 1)^{\alpha}
  \end{equation*}
  and, hence,
  \begin{equation}
    q_n^{\alpha^{- n}} \leq q_{n + 1}^{\alpha^{- (n + 1)}} < (q_{n + 1} +
    1)^{\alpha^{- (n + 1)}} < (q_n + 1)^{\alpha^{- n}} . \label{eq:mills:qn:a}
  \end{equation}
  It follows that the sequences $q_n^{\alpha^{- n}}$ and $(q_n + 1)^{\alpha^{-
  n}}$ are weakly increasing and decreasing, respectively. Being bounded by
  the terms of the latter, the sequence $q_n^{\alpha^{- n}}$ has a finite
  limit $A$ as $n \rightarrow \infty$. It follows from \eqref{eq:mills:qn:a}
  that
  \begin{equation*}
    q_n \leq A^{\alpha^n} < q_n + 1.
  \end{equation*}
  This is equivalent to $\lfloor A^{\alpha^n} \rfloor = q_n$.
\end{proof}

Note that the inequalities \eqref{eq:mills:qn:a} imply that
\begin{equation}
  q_n^{\alpha^{- n}} \leq A < (q_n + 1)^{\alpha^{- n}}
  \label{eq:mills:ct:est}
\end{equation}
and that this allows us to estimate the constant $A$ produced by the proof of
Lemma~\ref{lem:prc:legendre} to high precision.

\begin{example}
  \label{eg:mills:3}Assuming the RH, the case $\delta = 1$ of
  Theorem~\ref{thm:legendre:a:rh:intro} implies that the intervals $[n^3, (n +
  1)^3]$ and, hence, $[n^3, (n + 1)^3 - 1)$ contain a prime for all integers
  $n \geq 1$. In the proof of Lemma~\ref{lem:prc:legendre} we can
  therefore choose $q_1 = 2$ and then inductively select $q_{n + 1}$ to be the
  least prime exceeding $q_n^3$. Accordingly,
  \begin{equation*}
    q_2 = 11, q_3 = 1361, q_4 = 2521008887, \ldots
  \end{equation*}
  and we can estimate the resulting constant $A$ using
  \eqref{eq:mills:ct:est}. For $n = 4$, the two sides of
  \eqref{eq:mills:ct:est} differ by less than $6.4 \cdot 10^{- 12}$, thus
  producing the value of $A = 1.30637788 \ldots$ in \eqref{eq:mills:ct} with
  $11$ correct decimal digits. Using $n = 10$, Caldwell and Cheng
  \cite{cc-mills} computed $A$ to more than $6850$ digits. Since we made the
  minimal possible choices for each of the primes $q_n$, it is not hard to see
  that this value $A$ is the smallest real number such that $\lfloor A^{3^n}
  \rfloor$ is a prime for all $n$.
  
  We can proceed likewise for any exponent $\alpha = 2 + \delta$ where $\delta
  \geq 1 / 4$. Indeed, as indicated above, \ our proof of
  Theorem~\ref{thm:legendre:a:rh:intro} actually shows that, assuming the RH,
  the intervals $[x^{\alpha}, (x + 1)^{\alpha} - 1)$ contain a prime for all
  $x \geq 1$. In the proof of Lemma~\ref{lem:prc:legendre} we can
  therefore again choose $q_1 = 2$ and then inductively select $q_{n + 1}$ to
  be the least prime exceeding $q_n^{\alpha}$. For instance, for $\alpha =
  2.25$ this results in
  \begin{equation*}
    q_2 = 5, q_3 = 41, q_4 = 4259, q_5 = 146535287, q_6 =
     2362490078520203903, \ldots
  \end{equation*}
  as well as the corresponding value
  \begin{equation*}
    A = 1.385510850556131889057735009841 \ldots
  \end{equation*}
  for the smallest real number, assuming the RH, such that $\lfloor A^{2.25^n}
  \rfloor$ is a prime for all $n$.
\end{example}

By Lemma~\ref{lem:prc:legendre}, Legendre's conjecture implies that there is a
real number $A$ such that $\lfloor A^{2^n} \rfloor$ is a prime for all
positive integers $n$. Remarkably, Kaisa Matom{\"a}ki \cite{matomaki-prf}
was able to prove the existence of such a number $A$ without assuming
Legendre's conjecture. Her construction is different and more intricate than
the one in Lemma~\ref{lem:prc:legendre}, and does not provide a specific value
for $A$.

\begin{example}
  Assuming the Legendre conjecture, we can proceed as in
  Example~\ref{eg:mills:3} for $\alpha = 2$ in place of $3$ to construct and
  approximate the smallest $A$ such that $\lfloor A^{2^n} \rfloor$ is a prime
  for all $n$. That is, we let $q_1 = 2$ and then inductively define $q_{n +
  1}$ to be the least prime exceeding $q_n^2$. This results in
  \begin{equation*}
    q_2 = 5, q_3 = 29, q_4 = 853, q_5 = 727613, q_6 = 529420677791, \ldots
  \end{equation*}
  as well as the corresponding value
  \begin{equation*}
    A = 1.524699960538094359923363575688 \ldots .
  \end{equation*}
  By construction, this $A$ has the property that $\lfloor A^{2^n} \rfloor$ is
  a prime for all $n$ if we can prove the following.
\end{example}

\begin{conjecture}
  \label{conj:legendre:mills}Define the sequence $q_1, q_2, q_3, \ldots$ by
  $q_1 = 2$ and by letting $q_{n + 1}$ be the least prime exceeding $q_n^2$.
  Then $q_{n + 1} < (q_n + 1)^2$ for all $n$.
\end{conjecture}

Although Conjecture~\ref{conj:legendre:mills} (or a variation with the initial
prime $q_1 = 2$ replaced by another fixed prime) is weaker than Legendre's
conjecture, it is likely of a similar difficulty and thus still out of reach.

\section{Relaxing Legendre's conjecture in other directions.}

With a particular interest in the case $\alpha = 2$, we have discussed the
question of whether there is always a prime between $x^{\alpha}$ and $(x +
1)^{\alpha}$, possibly assuming that $x$ is large enough. For $\alpha > 2$,
there are weaker versions of Legendre's conjecture and
Theorem~\ref{thm:legendre:a:rh:intro} indicates what is currently known,
assuming the Riemann hypothesis.

Legendre's conjecture can be fruitfully weakened in other directions as well.
For instance, Danilo Bazzanella \cite{bazzanella-primes-sq} has proven that
Legendre's conjecture holds for most intervals $[n^2, (n + 1)^2]$ in the sense
that, for every $\varepsilon > 0$, each interval $[n^2, (n + 1)^2] \subseteq
[1, N]$, with at most $O (N^{1 / 4 + \varepsilon})$ exceptions, contains a
prime (and, in fact, the expected number of primes). Further stronger bounds
on the number of possible exceptions are provided in
\cite{bazzanella-primes-sq} and \cite{bazzanella-primes-sq2} under the
Riemann hypothesis as well as another weaker heuristic hypothesis. In a
similar direction, it is known that even smaller intervals ``almost always''
contain a prime. For instance, a result of Chaohua Jia \cite{jia-shortint}
implies that the intervals $[n, n + n^{\theta}]$ with $\theta = \frac{1}{20} +
\varepsilon$ contain a prime number for all $n \in [N, 2 N]$ with at most $o
(N)$ exceptions. By a recent result of Larry Guth and James Maynard
\cite{gm-lve} (Corollary~1.4) such intervals with $\theta = \frac{2}{15} +
\varepsilon$ almost always contain the expected number of primes. We refer the
interested reader to these papers for more details and history.

Approximating the Legendre conjecture in a different direction, Jing-Run Chen
\cite{chen-almostprimes} showed that there are integers with at most two
prime factors between sufficiently large consecutive squares. Recently, Dudek
and Johnston \cite{dj-almostprimes} showed that the restriction to
sufficiently large squares can be removed if one allows integers with at most
four prime factors.

{\textbf{Acknowledgements. }}
We thank the referees and editors for many careful and
helpful suggestions that improved this paper.

\end{document}